\documentclass[12pt,psamsfonts]{amsart}

\usepackage{amsmath,amssymb}
\usepackage{graphicx}
\usepackage[dvips]{psfrag}
\usepackage{color}
\usepackage{overpic} 
\usepackage[normalem]{ulem}%
\usepackage[dvipsnames]{xcolor}

\newtheorem{theorem}{Theorem}
\newtheorem{proposition}[theorem]{Proposition}

\newtheorem{example}{Example}

\newtheorem{mtheorem}{Theorem}

\newcommand{\p}{\partial}

\newcommand{\R}{{\mathbb R}}
\newcommand{\N}{{\mathbb N}}
\newcommand{\Ss}{{\mathbb S}}
\newcommand{\ep}{\varepsilon}
\newcommand{\te}{\theta}
\newcommand{\f}{\varphi}
\newcommand{\al}{\alpha}

\title[Melnikov analysis in nonsmooth differential systems
]{
Melnikov analysis in nonsmooth differential systems with nonlinear switching manifold}
\author[J. Bastos, C. Buzzi, J. Llibre, D. D. Novaes]{}


\allowdisplaybreaks[2]
\begin{document}
 \maketitle

\centerline{\scshape  J\'efferson L. R. Bastos, \; Claudio A. Buzzi}
\medskip

{\footnotesize \centerline{Universidade Estadual Paulista,
IBILCE-UNESP}\centerline{Av. Cristov\~ao Colombo, 2265, 15.054-000, S. J. Rio Preto, SP, Brasil }
\centerline{\email{jeferson@ibilce.unesp.br} and
\email{buzzi@ibilce.unesp.br}}}

\medskip

\centerline{\scshape Jaume Llibre}
\medskip

{\footnotesize \centerline{ Universitat Aut\`onoma de Barcelona, UAB} \centerline{Edifici C Facultat de Ci\`encies, 08193 Bellaterra, Barcelona, Spain}
\centerline{\email{jllibre@mat.uab.cat}}}

\medskip

\centerline{\scshape Douglas D. Novaes}
\medskip

{\footnotesize \centerline{Universidade Estadual de Campinas, IMECC-UNICAMP}\centerline{R. S\'{e}rgio Buarque de Holanda, 651, 13.083-859, Campinas, SP, Brasil}
\centerline{\email{ddnovaes@ime.unicamp.br}}}

\medskip

\bigskip

\begin{abstract}
We study the family of piecewise linear differential 
systems in the plane with two pieces separated by a cubic curve. Our 
main result is that 7 is a lower bound for the Hilbert number of this 
family. In order to get our main result, we develop the Melnikov functions 
for a class of nonsmooth differential systems, which generalizes, 
up to order 2, some previous results in the literature. Whereas the first order Melnikov 
function for the nonsmooth case remains the same as for
the smooth one (i.e. the first order averaged function) the second order Melnikov function for the nonsmooth case is different from the smooth one (i.e. the second order averaged function). We show that, in this case, a new term depending on the jump of discontinuity and on 
the geometry of the switching manifold is added to the second order averaged function.
\end{abstract}

\section{Introduction and statement of the main results}

In recent years, there has been a growing interest in studying nonsmooth differential systems, motivated mainly by
their engineering applications. In particular, piecewise linear differential systems have been used to model many real processes and different modern devices, for
more details see, for instance, \cite{BerBudChaKow2008} and the
references therein. In the smooth differential systems
many results have been derived under convenient smoothness
assumptions. Thus, the following natural questions arise: What
results from the smooth differential systems theory extend to the
nonsmooth ones? How we must modify the results from the smooth
differential systems theory in order that they work for the nonsmooth
ones? These questions are not merely academic because the
nonsmooth differential systems appear in a natural way
in the context of many applications, see for instance
\cite{ColJefLazOlm2017,MakLam2012,Sim2010,Tei2012}.

Here, we are mainly interested in studying the existence of limit cycles for piecewise linear differential systems with two pieces separated by a  nonlinear switching curve. The {\it Averaging Theory} is a classical method to attack this problem. It has been developed also for nonsmooth differential systems (see \cite{ILN2017,LliMerNov2015,LliNovRod2017,LNT15}). However, in the previous works it is assumed some strong conditions on the switching set, which do not hold when it is a cubic, for instance.  So here, we also develop the bifurcation functions of first and second order for computing periodic solutions of a wider class of nonsmooth differential systems. Usually, these bifurcation functions are called {\it Melnikov Functions}. Additionally,  we compare them with the Melnikov functions of first and second order for smooth systems.

Usually, piecewise linear differential systems have been considered when
a straight line separates the plane in two half--planes. In each of
the half--planes we have a linear differential system. If both
linear systems coincide in the switching line we call it a
continuous piecewise linear differential system. Otherwise we call
it a discontinuous piecewise linear differential system. In recent years, many authors have studied intensively the number of
limit cycles of discontinuous piecewise linear differential systems
with two zones separated by a straight line, see for instance \cite{ArtLliMedTei2013,BuzPesTor2013,BraMel2013,FrePonTor2012,FrePonTor2014,GiaPli2001,HanZha2010,HuaYan2013,HuaYan2013b,Li2014,LliNovTei2015b,LliNovTei2015,LliPon2012,LliTeiTor2013,LliZha2016} and the references quoted in these papers.
Up to now all results of these papers provide examples of at most 3
crossing limit cycles for this class of discontinuous systems. It
remains the open question: Is 3 the maximum number of crossing limit
cycles that such discontinuous differential systems can exhibit?

In this paper, we consider  the family $\mathcal{L}_h$ of piecewise linear differential systems in the
plane with two pieces separated by the curve $\Sigma=h^{-1}(0),$ $h(x,y)=y-x^3.$ More precisely,
we shall study the class  of discontinuous piecewise linear differential
systems obtained by perturbing up to order $N\in\N$ in the small
parameter $\varepsilon$ the linear center $\dot x=y,$ $\dot y=-x,$
i.e.
\begin{equation}\label{eq:1}
\begin{array}{ccl}
\dot x&=&\left\{ \begin{array}{l}y+\sum\limits_{i=1}^{N}\varepsilon^iP_{i}^{+}(x,y), \quad y\geq x^3\\ y+ \sum\limits_{i=1}^{N}\varepsilon^iP_{i}^{-}(x,y), \quad y\leq x^3\end{array}\right.\vspace{0.3cm}\\
\dot y&=&\left\{ \begin{array}{l}-x+\sum\limits_{i=1}^{N}\varepsilon^iQ_{i}^{+}(x,y), \quad y\geq x^3\\ -x+ \sum\limits_{i=1}^{N}\varepsilon^iQ_{i}^{-}(x,y), \quad y\leq x^3\end{array}\right.
\end{array}
\end{equation}
where the functions $P_{i}^{\pm}(x,y)$ and $Q_{i}^{\pm}(x,y)$ are the following
polynomials of degree one 
\[
\begin{array}{l}
P^+_{i}(x,y)=a_{0i}+a_{1i}x+a_{2i}y, \\
P^-_{i}(x,y)=b_{0i}+b_{1i}x+b_{2i}y, \\
Q^+_{i}(x,y)=\alpha_{0i}+\alpha_{1i}x+\alpha_{2i}y, \\
Q^-_{i}(x,y)=\beta_{0i}+\beta_{1i}x+\beta_{2i}y.
\end{array}
\]
We shall assume that $N=1,2.$

 As usual, we denote by $H_{\mathcal{L}_h}$ the maximum number of limit cycles that piecewise linear differential  systems in $\mathcal{L}_h$ can have. This number is called {\it Hilbert number}. Previous results show that $H_{\mathcal{L}_A}\geq 3,$ when $A:\R^2\rightarrow\R$ is a linear function. Our main result provides a lower bound for $H_{\mathcal{L}_h},$ when $h(x,y)=y-x^3.$

\begin{mtheorem}\label{main}
For $|\varepsilon|\neq0$ sufficiently small, there exist piecewise linear differential  systems of kind \eqref{eq:1} admitting at least $7$ limit cycles, which bifurcate from the periodic orbits of the linear center. Consequently, for $h(x,y)=y-x^3$ we have  $H_{\mathcal{L}_h}\geq 7.$
\end{mtheorem}

Theorem \ref{main} is proved in Section \ref{sec:proof}. Its proof is mainly based on the {\it Chebyshev Theory}, recently extended for {\it Chebyshev systems with positive accuracy} (see \cite{NovTor2017}), and {\it Melnikov Theory}. In order to prove Theorem \ref{main} we develop the Melnikov functions for a class of nonsmooth differential systems (see Theorem \ref{main2} in Section \ref{sec:mel}), which generalizes, up to order $2,$ the results of  \cite{ILN2017,LliNovRod2017}.

It is important to mention that at this moment in the literature
the maximum number of limit cycles exhibited by piecewise linear
differential systems in two zones separated by a straight line is
three. Theorem \ref{main} shows that the nonlinearity of the
switching curve $\Sigma$ is responsible for the
increasing of the number of limit cycles. In fact, this phenomena  has already been  observed in \cite{BraMel2014,NovPon2015}. Indeed, in \cite{NovPon2015} it was proved that the
number of limit cycles can increase arbitrarily with the number of oscillations of the switching curve $\Sigma.$ By oscillations of $\Sigma,$ one can think as the transversal intersections of $\Sigma$ with the straight line $x=0.$ For instance, in \cite{NovPon2015} the authors considered the switching curve $\Sigma=\{ (0,y) \mbox{ if }y\leq0\}\cup\{(h(y),y)\mbox{ if }y>0\},$ where 
\[h(y)=\left\{ \begin{array}{l} k \sin(\pi y), 0\leq y\leq \frac{2n+1}{2}, \\ \\ (-1)^n k, y>\frac{2n+1}{2}.\end{array}\right.\] 
 It was shown that the number of limit cycles of a particular piecewise linear differential  system corresponds to the number of zeros of the function $h,$ which corresponds to the intersections between $\Sigma$ and the straight line $x=0.$

\section{Melnikov functions for a class of nonsmooth systems}\label{sec:mel}

In order to state our results for the nonsmooth
differential systems, we need some notations and definitions. Let $D$
be an open set of $\R^d,$  $\Ss^1=\R/T$ for some period $T>0,$ and $N$ a positive integer. We
consider a finite set of $\mathcal{C}^r$ ($r\geq N+1$)  functions 
$\theta_i:D\rightarrow\Ss^1$ for $i=1,2,\dots,M,$ satisfying
$0<\theta_1(x)<\theta_2(x)<\cdots<\theta_M(x)<T$ for all $x\in D.$
For completeness, we take  $\theta_0(x)\equiv0$ and
$\theta_{M+1}(x)\equiv T.$

We consider the following \textit{nonsmooth differential
system}
\begin{equation}\label{eq:princ}
\dot x=\sum_{i=1}^N \ep^i F_i(t,x)+ \ep^{N+1}R(t,x,\ep),
\end{equation}
with
\[
F_i(t,x)=\left\{\begin{array}{ll}
F_i^0(t,x), & 0<t<\theta_1(x), \\
F_i^1(t,x), & \theta_1(x)<t<\theta_2(x), \\
\dots & \; \\
F_i^M(t,x), & \theta_M(x)<t<T,
\end{array}\right.
\]
and
\[
R(t,x,\ep)=\left\{\begin{array}{ll}
R^0(t,x,\ep), & 0<t<\theta_1(x), \\
R^1(t,x,\ep), & \theta_1(x)<t<\theta_2(x), \\
\dots & \; \\
R^M(t,x,\ep), & \theta_M(x)<t<T,
\end{array}\right.
\]
where $F_i^j:\Ss^1\times D\rightarrow\R^d,$ $R^j:\Ss^1\times
D\times(-\ep_0,\ep_0)\rightarrow\R^d,$ for $i=1,2,\ldots,N$ and
$j=0,1,\dots,M,$ are  $\mathcal{C}^r$ functions, $r\geq N+1,$ and $T$-periodic in the variable $t.$ Notice that, the switching manifold is given by $\Sigma=\{(\theta_i(x),x):\,x\in D,\,i=0,1,\ldots,M,M+1\}.$ Throughout the paper we shall also denote $F^j(t,x,\ep)=\sum_{i=1}^N \ep^i F_i^j(t,x)+ \ep^{N+1}R^j(t,x,\ep).$

The {\it Melnikov Theory} and the {\it Averaging Theory} are classical tools to investigate the existence of periodic solutions of perturbative systems of kind \eqref{eq:princ}. In short, both theories provide a sequence of functions $\Delta_i,$ $i=1,2,\ldots,k,$ which ``control'' the existence of isolated periodic solutions of \eqref{eq:princ}. These functions are called {\it Melnikov Functions} or {\it Bifurcation Functions}. In \cite{LliNovRod2017}, the averaging theory at any order was developed for systems of kind \eqref{eq:princ} assuming that the functions $\theta_i,$ $i=1,2,\ldots,M,$ were constant. It was shown that in this case the bifurcation functions coincide with the {\it averaged functions}, denoted by $f_i$'s, for smooth systems (see \cite{LNT2014}). In particular, for $i=1,2,$ we have
\begin{equation}\label{prom}
\begin{array}{l}
\displaystyle f_1(x)=\int_{0}^{T}F_1(s,x)ds\quad and\vspace{0.3cm}\\
\displaystyle  f_2(x)= \int_{0}^{T}\bigg[D_xF_1(s,x)
\int_{0}^{s}F_1(t,x)dt+F_2(s,x)\bigg]ds.
\end{array}
\end{equation}
Nevertheless, in \cite{LliMerNov2015} it was observed that higher order averaged functions, $f_i, i\geq 2,$ do not always control the bifurcation of isolated periodic solutions for nonsmooth differential systems, and a strong degenerate condition on the switching manifold was stablished in order that $f_2$ stands as the bifurcation function of order 2 (see hypotheses (Hb2) of \cite[Theorem B]{LliMerNov2015}).

As the main result of this section, Theorem \ref{main2} shows that the Melnikov functions of first and second order, $\Delta_1$ and $\Delta_2,$ of system \eqref{eq:princ} write
\begin{equation}\label{mel1}
\Delta_1(x)=f_1(x)\quad \text{and}\quad \Delta_2(x)= f_2(x) + f_2^*(x),
\end{equation}
where $f_1$ and $f_2$ are given by \eqref{prom}, and
\[
f_2^*(x)  = \sum_{j=1}^M \Big(F_1^{j-1}(\theta_j(x),x) -
F_1^{j}(\theta_j(x),x)\Big)
D_x\theta_j(x)\int_0^{\theta_j(x)}\hspace{-0.3cm}F_1(s,x)ds.
\]
This generalizes up to order 2 the results of \cite{ILN2017,LliNovRod2017}  for a wider class of nonsmooth differential equations. In what follows, $J\Delta_i(x)$ denotes the corresponding Jacobian matrix of $\Delta_i$ evaluated at $x.$

\begin{mtheorem} \label{main2}Consider the nonsmooth differential system \eqref{eq:princ} and the functions $\Delta_1$ and $\Delta_2$ defined in \eqref{mel1}. So, the following statements hold.
\begin{itemize} 
\item[{\bf i.}] {\bf (First Order)}
Assume that $a^\ast\in D$ satisfies $\Delta_1(a^\ast)=0$ and
$\det\left(J\Delta_1(a^*)\right)\neq0.$ Then, for $|\ep|\neq0$
sufficiently small, there exists a unique $T$-periodic solution
$x(t,\ep)$ of system \eqref{eq:princ} such that $x(0,\ep)\rightarrow
a^\ast$ as $\ep\rightarrow0.$

\smallskip

\item[{\bf ii.}] {\bf (Second Order)}
Assume that $\Delta_1(x)\equiv0$ and that $a^\ast\in D$ satifies
$\Delta_2(a^\ast)=0$ and $\det\left(J\Delta_2(a^*)\right)\neq0.$
Then, for $|\ep|\neq0$ sufficiently small, there exists a unique
$T$-periodic solution $x(t,\ep)$ of system \eqref{eq:princ} such
that $x(0,\ep)\rightarrow a^\ast$ as
$\ep\rightarrow0.$
\end{itemize}
\end{mtheorem}

From statement {\bf (i)} of Theorem \ref{main2}, we see that the first order Melnikov function $\Delta_1(x)$ for
nonsmooth differential systems coincides with the
corresponding one $f_1(x)$ for smooth differential systems. This had already been observed in \cite{LliMerNov2015}. The second order Melnikov function for smooth differential systems is equal to $f_2(x).$ Indeed, for smooth differential systems the Melnikov functions are, in some sense, equivalent to the averaged functions at any order (for more details, see \cite{HanRomZha2016}). Nevertheless, from statement {\bf (ii)} of Theorem \ref{main2}, the corresponding second order Melnikov function for nonsmooth differential systems of kind \eqref{eq:princ} has the extra term
$f_2^*(x),$ which depends on the jump of discontinuity and on the geometry of the switching manifold measured, respectively, by the expressions
\[
F_1^{j-1}(\theta_j(x),x) -
F_1^{j}(\theta_j(x),x)\quad \text{and}\quad D_x\theta_j(x)\int_0^{\theta_j(x)}\hspace{-0.3cm}F_1(s,x)ds.
\]

We notice that, in particular, if either the vector field $F_1$ is
continuous, or if the the switching manifold $\Sigma$ is ``vertical'', i.e. $D_x\theta_j(x)$ vanishes identically for $j=1,2,\ldots,M,$ then the extra term $f_2^*$ vanishes. Under one of these two last
assumptions Theorem \ref{main2} provides the known
results given in \cite{BuiLli2004,LNT2014} for smooth systems, and in \cite{ILN2017,LliMerNov2015,LliNovRod2017,LNT15} for nonsmooth differential systems.

In what follows, before the proof of Theorem \ref{main2}, we provide a general geometric interpretation for the increment $f_2^*.$ In short, the next result states that $f_2^*\neq0$ stands as a kind of transversal condition between the switching manifold $\Sigma$ and the average of the perturbation $F_1,$ which is highly related to hypotheses (Hb2) of \cite[Theorem B]{LliMerNov2015}. Accordingly, given $p=(t_p,x_p)\in\Sigma,$ we define the vector
\[
v(p)=\left(0,\dfrac{1}{t_p}\int_0^{t_p} F_1(t,x_p)dt\right).
\]
Notice that the second component of $v(p)$ is the average of the first order perturbation $F_1$ of \eqref{eq:princ} computed on the backward orbit $\{(t,x_p):t\in[0,t_p]\}$ of the unperturbed differential equation \eqref{eq:princ}$\big|_{\ep=0}$ with initial condition $(t(0),x(0))=p.$ 
\begin{proposition}\label{geo}
If $f_2^*\neq0,$ then there exists $p\in\Sigma$ such that $v(p)$ is transversal to $\Sigma$ at $p.$
\end{proposition}
\begin{proof}
We know that, for $p=(t_p,x_p)\in\Sigma,$ $t_p=\theta_j(x_p)$ for some $j\in\{0,1,\ldots,M,M+1\}.$ Notice that
\[
D_x\theta_j(x_p)\int_0^{\theta_j(x_p)}F_1(s,x)ds=0
\]
is equivalent to $v(p)\in T_p \Sigma.$ Consequently, $v(p)\in T_p \Sigma$ for every $p\in\Sigma$ implies that $f^*_2=0.$
\end{proof}

\begin{proof}[Proof of Theorem \ref{main2}]
We denote the solution of \eqref{eq:princ} by
\begin{equation}\label{piecewisesol}
\f(t,x,\ep)= \left\{\begin{array}{l}
     \f_0(t,x,\ep),\ 0\leq t\leq\al_1(x,\ep);\\
     \f_1(t,x,\ep),\ \al_1(x,\ep)\leq t\leq\al_2(x,\ep);\\
     \cdots \\
     \f_M(t,x,\ep),\ \al_M(x,\ep)\leq t\leq T.
\end{array}\right.
\end{equation}
In the above expression, $\al_j(x,\ep)$   is the flying time that the trajectory $\varphi_{j-1}(\cdot,x,\varepsilon),$  starting at $\varphi_{j-1}(\al_{j-1}(x,\ep),x,\varepsilon)\in D$ for $t=\al_{j-1}(x,\ep),$ reaches the manifold $\{(\theta_j(x),x):x\in D\}\subset\Sigma$ (see Figure \ref{fig0}), that is, 
\[
\al_j(x,\ep) =\te_j(\f_{j-1}(\al_j(x,\ep),x,\ep)).
\]
for $j=1,2,\dots,M,$  $\al_{0}(x,\ep)=0,$ and $\al_{M+1}(x,\ep)=T.$ Moreover,
 \begin{equation}\label{diff}
\begin{array}{l}
\dfrac{\p \f_j}{\p t}(t,x,\ep)=F^j(t,\f_j(t,x,\ep),\ep), \text{ for } j=0,1,\ldots, M, \text{ where}\vspace{0.2cm}\\

\left\{\begin{array}{l}
\f_0(0,x,\ep)=x,\vspace{0.1cm}\\
\f_j(\al_j(x,\ep),x,\ep)=\f_{j-1}(\al_j(x,\ep),x,\ep), \text{ for } j=1,2,\ldots, M.
\end{array}\right.
\end{array}
\end{equation}
From \eqref{diff} and from the differential dependence of the solutions on the initial conditions and on the parameter, we can see inductively that $\al_j(x,\ep)$ and  $\f_j(t,x,\ep)$ are $C^r$ functions, for $j=0,1,\ldots,M.$ 

\begin{figure}[h]
\begin{center}
\begin{overpic}[height=7cm]{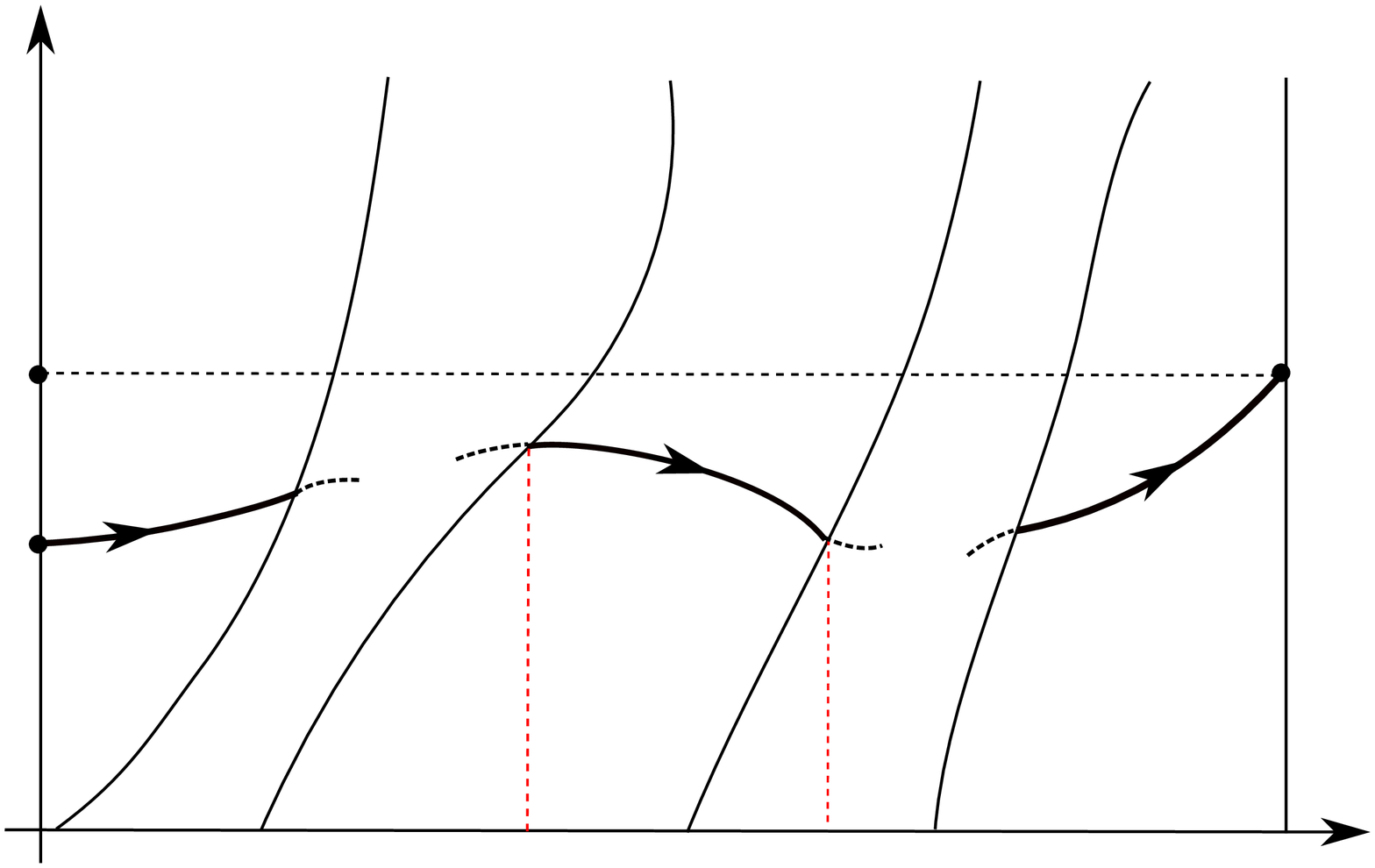} 
\put(-2,23){$x_0$} 
\put(2,64){$x$} 
\put(101,1.5){$t$} 
\put(95,37){$\f(T,x_0,\ep)$}
\put(-0.5,-1.5){$0$} 
\put(29,-1){\color{red}{$\al_{j-1}(x_0,\ep)$}} 
\put(54,-1){\color{red}{$\al_{j}(x_0,\ep)$}}
\put(25,59){$\te_{1}(x)$} 
\put(36,59){$\cdots$} 
\put(43.5,59){$\te_{j-1}(x)$} 
\put(67,59){$\te_{j}(x)$} 
\put(76.5,59){$\cdots$} 
\put(82,59){$\te_{M}(x)$} 
\put(93,-1.5){$T$}
\put(3.5,28){$\f_0(t,x_0,\ep)$}
\put(43,32){$\f_{j-1}(t,x_0,\ep)$}
\put(75,22){$\f_{M}(t,x_0,\ep)$}
\end{overpic}
\vspace{.3cm}
\caption{Illustration of a solution \eqref{piecewisesol} of system \eqref{eq:princ} starting at $x_0$ and crossing the switching manifold}\label{fig0}
\end{center}
\end{figure}

We notice that the recurrence \eqref{diff} describes initial value problems. Therefore, it is equivalent to the following recurrence:
\[\label{int}
\begin{array}{l}
\displaystyle \f_0(t,x,\ep)=x+\int_{0}^{t}F^0(s,\f_0(s,x,\ep),\ep)ds,\vspace{0.2cm}\\
\displaystyle \f_j(t,x,\ep)=\f_{j-1}(\al_{j}(x,\ep),x,\ep)+\int_{\al_{j}(x,\ep)}^{t}F^{j}(s,\f_j(s,x,\ep),\ep)ds,
\end{array}
\]
for $j=1,2,\ldots,M.$ So, we may compute
\[
\begin{aligned}
\f_M(t,x,\ep) =  x  & +\sum_{j=1}^{M}\int_{\al_{j-1}(x,\ep)}^{\al_j(x,\ep)}F^{j-1}(s,\f_{j-1}(s,x,\ep),\ep)ds \\
& +\int_{\al_M(x,\ep)}^{t}F^M(s,\f_M(s,x,\ep),\ep)ds.
\end{aligned}
\]

The displacement function is given by
\begin{equation}\label{eq:melnikov}
\Delta(x,\ep)=\f(T,x,\ep)-x=\f_M(T,x_0,\ep)-x.
\end{equation}
Notice that, from the above comments, $\Delta(x,\ep)$ is a $C^r$ function. Then, expanding \eqref{eq:melnikov} around $\ep=0$ we have
\[
\Delta(x,\ep)=\Delta_0(x)+ \ep\Delta_1(x)+ \ep^2\Delta_2(x)+\mathcal{O}(\ep^3).
\]
We observe that $\Delta_0(x) \equiv 0,$ because $\Delta_0(x) = \Delta(x,0) = \f(T,x,0)-x.$
The functions $\Delta_1(x)$ and $\Delta_2(x)$ are called Melnikov functions of first and second order, respectively.

The displacement function reads
\begin{equation}\label{eq:14}
\Delta(x,\ep)=\f_M(T,x,\ep) - x  =  \sum_{j=1}^{M+1}\int_{\al_{j-1}(x,\ep)}^{\al_j(x,\ep)}\hspace{-.3cm}F^{j-1}(s,\f_{j-1}(s,x,\ep),\ep)ds.
\end{equation}
In what follows, we compute the expansion, around $\ep=0,$ of the $M+1$ summands of \eqref{eq:14}. For the first one we obtain
\begin{equation}\label{eq:15}
\begin{aligned} &\int_{0}^{\al_1(x,\ep)}F^{0}(s,\f_0(s,x,\ep),\ep)ds \\ &=\int_{0}^{\al_1(x,\ep)}[\ep F_1^{0}(s,\f_0(s,x,\ep))+ \ep^2 F_2^{0}(s,\f_0(s,x,\ep))+ \mathcal{O}(\ep^3)]ds\\
&=\ep \Bigg(\int_{0}^{\te_1(x)}F_1^{0}(s,x)ds\Bigg)+ \ep^2\Bigg(\int_{0}^{\te_1(x)}\bigg[DF_1^0(s,x)\frac{\partial \f_0}{\partial  \ep}(s,x,0)\\ &+F_2^0(s,x)\bigg]ds+F_1^0(\te_1(x),x)\frac{\partial \al_1}{\partial \ep}(x,0)\Bigg)+ \mathcal{O}(\ep^3).
\end{aligned}
\end{equation}
In order to finish the computation of the second summand we have to compute the terms $\frac{\partial \f_0}{\partial  \ep}(s,x,0)$ and $\frac{\partial \al_1}{\partial\ep} (x,0).$ Expanding the equation
\[
\f_0(t,x,\ep)=x+\int_{0}^{t}F^{0}(s,\f_0(s,x,\ep),\ep)ds
\]
around $\ep=0$ we obtain
\[
x+\ep \frac{\partial \f_0}{\partial\ep}(t,x,0)+ \mathcal{O}(\ep^2)= x+\ep
\int_{0}^{t}F^{0}_1(s,x)ds+
\mathcal{O}(\ep^2).
\]
So, we get
\begin{equation}\label{eq:dfi1}
\frac{\partial \f_0}{\partial
    \ep}(t,x,0)=\int_{0}^{t}F_1^{0}(s,x)ds.
\end{equation}
In the same way, expanding the equation
$\al_1(x,\ep)=\te_1(\f_0(\al_1(x,\ep),x,\ep))$
around $\ep=0$ we obtain
\[
\begin{aligned}
& \al_1(x,0)+\ep \frac{\partial \al_1}{\partial \ep}(x,0)+
\mathcal{O}(\ep^2) = \\
& \te_1(x)+ \ep D_x\te_1(x)\bigg(\frac{\partial \f_0}{\partial
    t}(\te_1(x),x,0)\frac{\partial \al_1}{\partial \ep}(x,0)+\frac{\partial \f_0}{\partial
    \ep}(\te_1(x),x,0)\bigg)+\mathcal{O}(\ep^2).
\end{aligned}
\]
Since $\f_0(t,x,0)=x$ for all $t$ we get
\begin{equation}\label{eq:dalpha1}
\frac{\partial \al_1}{\partial
    \ep}(x,0)=D_x\te_1(x)\int_{0}^{\te_1(x)}F_1^{0}(s,x)ds.
\end{equation}
Substituting \eqref{eq:dfi1} and \eqref{eq:dalpha1} in \eqref{eq:15} we have the expression for the first summand of \eqref{eq:14}
\begin{equation}\label{parc1}
\begin{aligned}
&\int_{0}^{\al_1(x,\ep)}F^{0}(s,\f_0(s,x,\ep),\ep)ds  =
\ep \Bigg(\int_{0}^{\te_1(x)}F_1^{0}(s,x)ds\Bigg) \\
&+\ep^2\Bigg(\int_{0}^{\te_1(x)}\bigg[D_xF_1^0(s,x) \int_{0}^{s}F_1^{0}(t,x)dt+F_2^0(s,x)\bigg]ds\\ &+F_1^0(\te_1(x),x)D_x\te_1(x)\int_{0}^{\te_1(x)}F_1^{0}(s,x)ds\Bigg)+ \mathcal{O}(\ep^3).
\end{aligned}
\end{equation}
The other summands of \eqref{eq:14} can be computed in a similar way and they are given by
\begin{equation}\label{parc2}
\begin{aligned}
&\int_{\al_{j-1}(x,\ep)}^{\al_{j}(x,\ep)}F^{j-1}(s,\f_{j-1}(s,x,\ep),\ep)ds  =
\ep \Bigg(\int_{\te_{j-1}(x)}^{\te_j(x)}F_1^{j-1}(s,x)ds\Bigg) \\
&+\ep^2\Bigg(\int_{\te_{j-1}(x)}^{\te_j(x)}\bigg[D_xF_1^{j-1}(s,x) \int_{0}^{s}F_1(t,x)dt+F_2^{j-1}(s,x)\bigg]ds\\ &+F_1^{j-1}(\te_j(x),x)D_x\te_j(x)\int_{0}^{\te_j(x)}F_1(s,x)ds \\ &-F_1^{j-1}(\te_{j-1}(x),x)D_x\te_{j-1}(x)\int_{0}^{\te_{j-1}(x)}F_1(s,x)ds\Bigg)+ \mathcal{O}(\ep^3).
\end{aligned}
\end{equation}
From \eqref{parc1} and \eqref{parc2}, we have that the Melnikov functions of first and second order are given by
\[
\Delta_1(x)= \int_{0}^{T}F_1(s,x)ds
\]
and
\[
\Delta_2(x)= f_2(x) + f_2^*(x),
\]
where
\[
f_2(x)= \int_{0}^{T}\bigg[D_xF_1(s,x) \int_{0}^{s}F_1(t,x)dt+F_2(s,x)\bigg]ds
\]
and
\[
f_2^*(x)  = \sum_{j=1}^M \bigg(F_1^{j-1}(\theta_j(x),x)-F_1^j(\theta_j(x),x)\bigg)D_x\theta_j(x)\int_0^{\theta_j(x)}\hspace{-0.3cm}F_1(s,x)ds.
\]

From the definition of the displacement function in \eqref{eq:melnikov}, it is clear that the $T$-periodic solutions $x(t,\ep)$ of system \eqref{eq:princ}, satisfying $x(0,\ep)=x,$ are in one-to-one correspondence to the zeros of the equation $\Delta(x,\ep)=0.$

Now, define $\widehat \Delta(x,\ep)=\Delta(x,\ep)/\ep=\Delta_1(x)+\mathcal{O}(\ep).$ From hypotheses, we have that
\[
\widehat\Delta(a^*,0)=\Delta_1(a^*)=0, \text{ and } \det\left(\dfrac{\p\widehat\Delta}{\p x}(a^*,0)\right)=\det\left(J \Delta_1(a*)\right)\neq0.
\]
Therefore, from the {\it Implicit Function Theorem}, we get the
existence of a unique $C^r$ function $a(\ep)\in D$ such that
$a(0)=a^*$ and $\Delta(a(\ep),\ep)=\widehat \Delta(a(\ep),\ep)=0,$
for every $|\ep|\neq0$ sufficiently small. This completes the proof of statement {\bf (i)}.

Finally, assuming that $\Delta_1(x)\equiv0,$ define $\widetilde \Delta(x,\ep)=\Delta(x,\ep)/\ep^2=\Delta_2(x)+\mathcal{O}(\ep).$
So, the proof of statement {\bf (ii)} follows analogously to the proof of statement {\bf (i)}. 
\end{proof}

\section{Bifurcation of limit cycles}\label{sec:proof}
 In this section we shall apply the Melnikov functions developed in the previous section to study the bifurcation of limit cycles of system \eqref{eq:1} for $|\varepsilon|\neq0$ sufficiently small. Then, Theorem \ref{main} will be a direct consequence of Proposition \ref{prop2}.

A useful tool to study the number of isolated zeros of a function is the {\it Chebyshev Theory}. We recall that the Wronskian of the ordered $k+1$ functions $u_0, \ldots, u_k$ is
\begin{equation}\label{Wron}
W_k(x)=W_k(u_0,\ldots,u_k)(x)=\det\big(M(u_0,\ldots,u_k)(x)\big),
\end{equation}
where
\[
M(u_0,\ldots,u_k)(x)=\left(\begin{array}{ccc}
u_0(x)&\cdots&u_k(x)\\
u_0'(x)&\cdots&u_k'(x)\\
\vdots&\ddots&\vdots\\
u_0^{(k)}(x)&\cdots& u_k^{(k)}(x)
\end{array}\right).
\]
We say that $\mathcal{F}=[u_0,u_1,\ldots,u_n]$ is an {\it Extended Complete Chebyshev} system or an ECT-system on a closed interval $[a,b]$ if and only if $W(u_0,u_1,\ldots,u_k)(x)\neq 0$ on $[a,b]$ for $0\leq k\leq n,$ see \cite{KarStu1966}. In what follows, $\textrm{Span}(\mathcal{F})$ denotes the set of all  functions given by linear combinations of the functions of  $\mathcal{F}.$ A classical result concerning ECT-system is the following:

\begin{theorem}[\cite{KarStu1966}]\label{ECT}
Let $\mathcal{F}=[u_0,u_1,\ldots,u_n]$ be an ECT-system on  a closed interval $[a,b].$ Then, the number of isolated zeros for every element of $\textrm{Span}(\mathcal{F})$ does not exceed $n.$ Moreover, for each configuration of $m\le n$ zeros, taking into account their multiplicity,
there exists $F\in\textrm{Span}(\mathcal{F})$ with this configuration of zeros.
\end{theorem}

In \cite{NovTor2017}, the above result was extended for Chebyshev systems with positive accuracy. In what follows, we say that a zero of a function is {\it simple} if its Jacobian matrix evaluated at the zero is a nonsingular matrix.

\begin{theorem}[\cite{NovTor2017}]\label{ECT2}
Let $\mathcal{F}=[u_0,u_1,\ldots,u_n]$ be an ordered set of functions on $[a,b].$ Assume that all the
Wronskians $W_k(x),$ $k=0,\dots,n-1,$ are nonvanishing except $W_n(x),$ which has exactly one zero on $(a, b)$ and this zero is simple. Then, the number of isolated zeros for every element of $\textrm{Span}(\mathcal{F})$ does not exceed $n+1.$ Moreover, for any configuration of $m\le n+1$ zeros there exists $F\in\textrm{Span}(\mathcal{F})$
realizing it.
\end{theorem}

As a first application of Theorem \ref{main2} we consider piecewise
linear differential systems \eqref{eq:1} with $N=1.$ The next result shows that in this case the first order Melnikov function $\Delta_1$ has at most 3 isolated zeros. From statement {\bf(i)} of Theorem \ref{main2}, this means that the differential system \eqref{eq:1} has at most 3 limit cycles bifurcating from the periodic orbits of the linear center. Moreover, we shall see that this maximum number is reached by particular examples.

 \begin{proposition}\label{prop1}
Consider the piecewise linear differential system \eqref{eq:1} with $N=1,$ and let $\Delta_1$ be its first order Melnikov function \eqref{mel1}. Then,  the number of isolated zeros of $\Delta_1$ does not exceed $3.$ Moreover, there exist linear differential systems \eqref{eq:1} such that $\Delta_1$ has exactly 3 simple zeros.
\end{proposition}
\begin{proof}
In order to apply the Melnikov method, we must write the differential system \eqref{eq:1} in polar coordinates. Performing the change of variables $(x,y)=(r\cos\theta,r\sin\theta)$ and taking $\theta$ as the new independent variable, the differential system \eqref{eq:1} is transformed into the
differential equation
\[\label{eq:pol}
\frac{dr}{d\theta} = \varepsilon F_1^\pm(\theta,r) + \varepsilon^2 F_2^\pm(\theta,r)+\mathcal{O}(\varepsilon^3),
\]
where
\[\label{eq:f1p}
\begin{array}{ll}
F_1^+(\theta,r)=& - (a_{01}\cos\theta +b_{01}\sin\theta) - \frac{r}{2} \big(a_{11} +b_{21} \big) - \\ & \frac{r}{2} \big((a_{11}-b_{21}) \cos(2 \theta) + (a_{21} +b_{11})\sin(2 \theta)\big)
\end{array}
\]
and
\[\label{eq:f1m}
\begin{array}{ll}
F_1^-(\theta,r)=& - (\alpha_{01}\cos\theta +\beta_{01}\sin\theta) - \frac{r}{2} \big(\alpha_{11} +\beta_{21}\big)- \\ &
\frac{r}{2} \big((\alpha_{11}-\beta_{21}) \cos(2 \theta) + (\alpha_{21} +\beta_{11})\sin(2 \theta)\big).
\end{array}
\]

The switching curve  $\Sigma$ is given in polar coordinates by $\sin\theta=r^2\cos^3\theta.$ It means that for each $r>0$ there exist two angles $\theta_1(r)\in[0,\frac{\pi}{2}]$ and
$\theta_2(r)=\theta_1(r)+\pi\in[\pi,\frac{3\pi}{2}]$ such that the circle $x^2+y^2=r^2$ intersects $\Sigma$ at the
points $(r\cos\theta_1(r),r\sin\theta_1(r))$ and $(r\cos\theta_2(r),r\sin\theta_2(r))$ (see Figure \ref{fig1}).

\begin{figure}[h]
    \begin{center}
        \begin{overpic}[height=6cm]{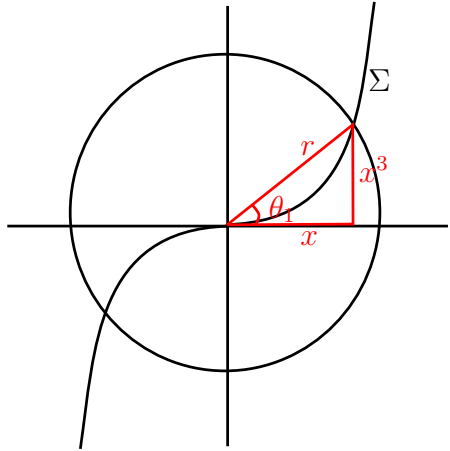}
            \put(58,52){\color{red}{$\theta_1$}}
            \put(65,46){\color{red}{$x$}}
            \put(78,60){\color{red}{$x^3$}}
            \put(65,66){\color{red}{$r$}}
            \put(80,80){$\Sigma$}
        \end{overpic}
        \caption{Identities involving $x,$ $r$ and $\theta_1.$}\label{fig1}
    \end{center}
\end{figure}

According to the theory developed in Section \ref{sec:mel}, the Melnikov function of first order is given by
\[\begin{array}{ll}
 \Delta_1(r) & = \displaystyle \frac{1}{2\pi}\left(\int_0^{\theta_1(r)} \hspace{-.5cm} F_1^-(\theta,r)d\theta + \int_{\theta_1(r)}^{\theta_1(r)+\pi}
  \hspace{-.5cm} F_1^+(\theta,r)d\theta + \int_{\theta_1(r)+\pi}^{2\pi} \hspace{-.5cm} F_1^-(\theta,r)d\theta\right) \\ \\
& = \gamma_0 \cos\theta_1(r)+\gamma_1 r + \gamma_2\sin\theta_1(r),
\end{array}\]
where
\begin{equation}\label{eq:gamas}
\begin{array}{l}
\gamma_0=-2 (b_{01}-\beta_{01}), \\
\gamma_1=\frac{-\pi}{2} (a_{11}+b_{21}+\alpha_{11}+\beta_{21}), \\
\gamma_2= a_{01}-\alpha_{01}.
\end{array}
\end{equation}

From Theorem \ref{main2}, each isolated zero of $\Delta_1(r)=0$ corresponds to a limit cycle of \eqref{eq:1}.
We prove now that the maximum number of positive zeros $\Delta_1(r)=0$ is three. As we can see in Figure \ref{fig1}, the variables $x,$ $r$ and $\theta_1$ are related by the identities
\begin{equation}\label{eq:id}
r^2=x^2+x^6, \quad \cos\theta_1=\frac{x}{r}, \quad\mbox{ and }\quad \sin\theta_1=\frac{x^3}{r}.
\end{equation}
For $r>0$ we have that $\Delta_1(r)=0$ if and only if $r\Delta_1(r)=0.$ So, the maximum number of limit cycles of \eqref{eq:1} obtained by the first order Melnikov method is the maximum number of positive zeros of the polynomial
\begin{equation}\label{eq:xx}
p_1(x)=\gamma_0 + \gamma_1(x+x^5) + \gamma_2 x^2.
\end{equation}
Consider the ordered set of functions $\mathcal{F}=[u_0,u_1,u_2],$ where $u_0(x)=1,$ $u_1(x)=x+x^5$ and $u_2(x)=x^2.$ We have that the Wronskians, defined in \eqref{Wron}, are $W_0(x)=1,$ $W_1(x)=1+5x^4$ and $W_2(x)=2-30x^4.$ The Wronskians $W_0(x)$ and $W_1(x)$ do not vanish in $\R$ and $W_2(x)$ has exactly one positive zero. So, according to Theorem \ref{ECT2}, three is the maximum number of positive roots of polynomial \eqref{eq:xx} and there is a choice for $\gamma_0,$ $\gamma_1$ and $\gamma_2$ such that $p_1(x)$ has exactly three zeros. From \eqref{eq:gamas}, it is clear that for any choice of $\gamma$'s it is possible to obtain the corresponding perturbation. This concludes the proof of Proposition \ref{prop1}.
\end{proof}

\begin{example}
In what follows, we exhibit a concrete perturbation presenting three limit cycles for $|\ep|\neq0$ sufficiently small. For $\gamma_0=-0.05,$ $\gamma_1=1$  and $\gamma_2=-2,$ the polynomial \eqref{eq:xx} has three positive roots (see Figure \ref{fig2}).
\begin{figure}[h]
\begin{center}
\begin{overpic}[height=5cm]{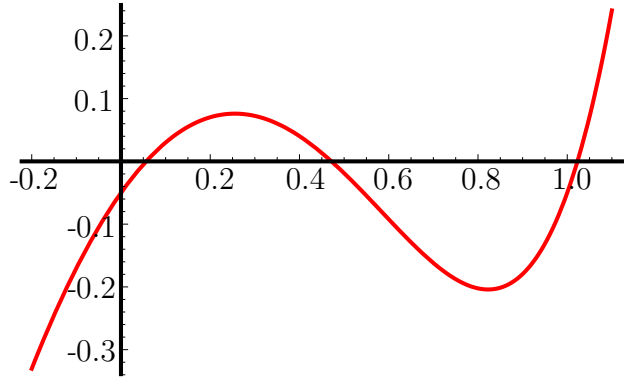} 
    \put(9,54){$0.2$} \put(9,44){$0.1$} \put(7.5,23){-$0.1$} \put(7.5,13){-$0.2$} \put(7.5,2.5){-$0.3$} \put(-2,31){-$0.2$} \put(29,31){$0.2$} \put(44,31){$0.4$} \put(58.5,31){$0.6$} \put(73,31){$0.8$} \put(88,31){$1.0$}
\end{overpic}
\caption{Graph of $p_1(x)= x^5 -2 x^2 + x -0.05.$}\label{fig2}
\end{center}
\end{figure}
It is clear that the linear system
\[\begin{array}{l}
-2 (b_{01}-\beta_{01})=-0.05,\\
\frac{-\pi}{2}(a_{11}+b_{21}+\alpha_{11}+\beta_{21})=1,\\
a_{01}-\alpha_{01}=-2,
\end{array}\]
has infinity many solutions. For example, we can take $a_{01}=-2,$ $a_{11}=-2/\pi,$ $b_{01}=1/40$ and all other coefficients equal to zero. So, for $|\ep|\neq0$ sufficiently small, system
\[\begin{array}{l}
\dot x = \left\{\begin{array}{ll}
y - 2\varepsilon (1+\frac{1}{\pi}x), & \mbox{ if } y\geq x^3, \\
y +\frac{1}{40}\varepsilon, & \mbox{ if } y\leq x^3,
\end{array} \right. \\
\dot y = -x
\end{array}
\]
has three limit cycles.
\end{example}

As a second application of Theorem \ref{main2}, we consider piecewise linear differential systems \eqref{eq:1} with $N=2.$ We impose conditions on the coefficients of $P_1^\pm$ and $Q_1^\pm$ in order that the first order Melnikov function $\Delta_1$ vanishes identically. The next result shows that in this case the second order Melnikov function $\Delta_2$ has at most 7 isolated zeros. From statement {\bf(ii)} of Theorem \ref{main2}, this means that the differential system \eqref{eq:1} has at most 7 limit cycles bifurcating from the periodic orbits of the linear center. Moreover, we shall see that this maximum number is reached by particular examples. This proves Theorem \ref{main}.

 \begin{proposition}\label{prop2}
Consider the piecewise linear differential system \eqref{eq:1} with $N=2,$ and let $\Delta_1$ and $\Delta_2$ be its first and second order Melnikov functions \eqref{mel1}, respectively. In addition, assume that $\Delta_1=0.$ Then,  the number of isolated zeros of $\Delta_2$ does not exceed 7. Moreover, there exist linear differential systems \eqref{eq:1} such that $\Delta_2$ has exactly 7 simple zeros.
\end{proposition}

\begin{proof}
In order to apply the Melnikov method of second order it is necessary that the Melnikov function of first order $\Delta_1$ be identically zero.
So, we impose that the $\gamma$'s defined in \eqref{eq:gamas} are identically zero. It is enough to take
\begin{equation} \label{eq:con1}
\begin{array}{l}
a_{11}=-(b_{21}+\alpha_{11}+\beta_{21}), \\
b_{01}=\beta_{01}, \\
a_{01}= \alpha_{01}.
\end{array}
\end{equation}
Again, we use polar coordinates $(x,y)=(r\cos\theta,r\sin\theta)$ and take $\theta$ as the new independent variable to transform system \eqref{eq:1} into the differential equation
\[
\frac{dr}{d\theta} = \varepsilon F_1^\pm(\theta,r) + \varepsilon^2 F_2^\pm(\theta,r) + \varepsilon^3 R_2^\pm(\theta,r,\varepsilon),
\]
where
\[
\begin{array}{l}
F_1^\pm(\theta,r)=  f_{11}^\pm(\theta) + r f_{12}^\pm(\theta), \\ [5pt]
F_2^\pm(\theta,r)=  \displaystyle\frac{1}{r} f_{21}^\pm(\theta) + f_{22}^\pm(\theta) + rf_{23}^\pm(\theta),\end{array}\]
with
\begin{align*}
f_{11}^\pm(\theta)&= c_{101}^\pm\cos\theta+s_{101}^\pm\sin\theta, \\
f_{12}^\pm(\theta)&= c_{110}^\pm +c_{112}^\pm\cos(2\theta)+s_{112}^\pm\sin(2\theta), \\
f_{21}^\pm(\theta)&=c_{202}^\pm\cos(2\theta)+s_{202}^\pm\sin(2\theta), \\
f_{22}^\pm(\theta)&=c_{211}^\pm\cos\theta +s_{211}^\pm\sin\theta +c_{213}^\pm\cos(3\theta)+s_{213}^\pm\sin(3\theta), \\
f_{23}^\pm(\theta)&=c_{220}^\pm +c_{222}^\pm\cos(2\theta) + c_{224}^\pm\cos(4\theta) + s_{222}^\pm\sin(2\theta) + s_{224}^\pm\sin(4\theta),
\end{align*}
and
\begin{align*}
c_{101}^+&= -\alpha_{01}, \hspace{1cm} s_{101}^+= -\beta_{01}, \hspace{1cm}
c_{110}^+= \frac{1}{2}(\alpha_{11} + \beta_{21}),   \\
c_{112}^+&= \frac{1}{2} (2 b_{21}+\alpha_{11}+\beta_{21}), \hspace{1cm}  s_{112}^+= \frac{1}{2} (-a_{21}-b_{11}), \\
c_{202}^+&= -\alpha_{01} \beta_{01},  s_{202}^+= \frac{1}{2} \left(\alpha_{01}^2-\beta_{01}^2\right), \\
c_{211}^+&= \frac{1}{2} (-2a_{02}+a_{21} \alpha_{01}-b_{11}\alpha_{01}+\alpha_{11}\beta_{01}+\beta_{01}\beta_{21}), \\
s_{211}^+&= \frac{1}{2} (a_{21}\beta_{01}-2 b_{02}-b_{11}\beta_{01}-\alpha_{01}\alpha_{11}-\alpha_{01}\beta_{21}), \\
c_{213}^+&= \frac{1}{2} (-a_{21}\alpha_{01}-b_{11}\alpha_{01}+2b_{21}\beta_{01}+\alpha_{11}\beta_{01}+\beta_{01}\beta_{21}), \\
s_{213}^+&= \frac{1}{2} (-a_{21}\beta_{01}-b_{11}\beta_{01}-2b_{21}\alpha_{01}-\alpha_{01}\alpha_{11}-\alpha_{01}\beta_{21}), \\
c_{220}^+&= \frac{1}{4} (-2a_{12}-a_{21}\alpha_{11}-a_{21}\beta_{21}+b_{11}\alpha_{11}+b_{11}\beta_{21}-2b_{22}),  \\
c_{222}^+&= \frac{1}{2} (-a_{12}-a_{21}b_{21}+b_{11}b_{21}+b_{11}\alpha_{11} +b_{11}\beta_{21}+b_{22}), \\
s_{222}^+&= \frac{1}{4} \left(a_{21}^2-2a_{22}-b_{11}^2-2b_{12}+2b_{21}\alpha_{11}+ 2b_{21}\beta_{21}+(\alpha_{11}+\beta_{21})^2\right), \\
c_{224}^+&= \frac{1}{4} (a_{21}+b_{11}) (2b_{21}+\alpha_{11}+\beta_{21}), \\
s_{224}^+&= \frac{1}{8} \left(-a_{21}^2-2a_{21}b_{11}-b_{11}^2+(2b_{21}+\alpha_{11} +\beta_{21})^2\right), \\
c_{101}^-&=-\alpha_{01}, \hspace{1cm} s_{101}^-=  -\beta_{01}, \hspace{1cm}
c_{110}^-= \frac{1}{2} (-\alpha_{11}-\beta_{21}),   \\
c_{112}^-&= \frac{1}{2}(\beta_{21}-\alpha_{11}),  \hspace{1cm}
s_{112}^-= \frac{1}{2} (2b_{21}+\alpha_{11}+\beta_{21}), \hspace{1cm} \\
c_{202}^-&=-\alpha_{01}\beta_{01}, \hspace{1cm}
s_{202}^-= \frac{1}{2} \left(\alpha_{01}^2-\beta_{01}^2\right), \\
c_{211}^-&= \frac{1}{2} (\alpha_{01}\alpha_{21}-\alpha_{01}\beta_{11}-2\alpha_{02}- \alpha_{11}\beta_{01}-\beta_{01}\beta_{21}), \\
s_{211}^-&= \frac{1}{2} (\alpha_{01} (\alpha_{11}+\beta_{21})+\alpha_{21}\beta_{01}- \beta_{01}\beta_{11}-2\beta_{02}), \\
c_{213}^-&= \frac{1}{2} (\beta_{01} (\beta_{21}-\alpha_{11})-\alpha_{01}(\alpha_{21}+ \beta_{11})), \\
s_{213}^-&= \frac{1}{2} (\alpha_{01} (\alpha_{11}-\beta_{21})-\beta_{01}
(\alpha_{21}+\beta_{11})), \\
c_{220}^-&= \frac{1}{4} (\alpha_{11}\alpha_{21}-\alpha_{11}\beta_{11}-2\alpha_{12}+ \alpha_{21}\beta_{21}-\beta_{11}\beta_{21}-2\beta_{22}),  \\
c_{222}^-&= \frac{1}{2} (-\alpha_{11}\beta_{11}-\alpha_{12}-\alpha_{21}\beta_{21}+ \beta_{22}), \\
s_{222}^-&= \frac{1}{4} \left(\alpha_{11}^2+\alpha_{21}^2-2\alpha_{22}-\beta_{11}^2- 2\beta_{12}-\beta_{21}^2\right), \\
c_{224}^-&= -\frac{1}{4} (\alpha_{11}-\beta_{21}) (\alpha_{21}+\beta_{11}), \\
s_{224}^-&=\frac{1}{8} \left(\alpha_{11}^2-2\alpha_{11}\beta_{21}-\alpha_{21}^2- 2\alpha_{21}\beta_{11}-\beta_{11}^2+\beta_{21}^2\right).
\end{align*}
According to Section \ref{sec:mel} the Melnikov function of second order $\Delta_2$ is given by
\[\label{melnikov2}
\Delta_2(r)= f_2(r) + f_2^*(r),
\]
where
\[\begin{array}{lcl}
f_2(r)&=&\displaystyle\frac{1}{2\pi}\int_0^{\theta_1(r)}\!\!\left[D_rF_1^-(\theta,r) \int_0^\theta F_1(t,r)dt+F_2^-(\theta,r)\right]d\theta  \\ [10pt]
&+&\displaystyle\frac{1}{2\pi}\int_{\theta_1(r)}^{\theta_2(r)}\!\!\left[D_rF_1^+(\theta,r)\int_0^\theta F_1(t,r)dt+F_2^+(\theta,r)\right]d\theta  \\ [10pt]
&+&\displaystyle\frac{1}{2\pi}\int_{\theta_2(r)}^{2\pi}\,\left[D_rF_1^-(\theta,r)\int_0^\theta F_1(t,r)dt+F_2^-(\theta,r)\right]d\theta.
\end{array}\]
and
\[
\begin{aligned}
f_2^*(r) & =\Big(F_1^-(\te_1(r),r)-F_1^+(\te_1(r),r)\Big)\te_1^\prime(r)\int_0^{\te_1(r)}F_1(s,r)ds \\
&+\Big(F_1^+(\te_2(r),r)-F_1^-(\te_2(r),r)\Big)\te_2^\prime(r)\int_0^{\te_2(r)}F_1(s,r)ds.
\end{aligned}
\]
We recall that in our context  $\theta_2(r)=\theta_1(r)+\pi.$ So, we have $\te_1^\prime(r)=\te_2^\prime(r).$
For simplicity in the following formulas we denote $\theta_1(r)$  by $\theta_1.$ After some computations we obtain
\[
f_2(r)=\delta_1r+\delta_2\cos \theta_1+\delta_3\sin \theta_1+\delta_4\cos^3\theta_1+\delta_5\sin^3\theta_1,
\]
and
\[
f_2^*(r)=\frac{r^2\cos^2\theta_1}{2+6r^2\cos\theta_1\sin\theta_1}\big[\mu_1 r+ \mu_2\cos\theta_1+\mu_3\sin\theta_1\big]h(\theta_1),
\]
where
\begin{align*}
h(\theta_1)&=\eta_1+\eta_2\cos^2\theta_1+\eta_3\cos\theta_1\sin\theta_1,\\
\delta_1&= -\frac{\pi}{4} \big( 2 a_{12} + 2 b_{22} + a_{21} \alpha_{11} - b_{11} \alpha_{11} + 2 \alpha_{12} - \alpha_{11} \alpha_{21} \\
&+ \alpha_{11} \beta_{11} + a_{21} \beta_{21} - b_{11} \beta_{21} - \alpha_{21} \beta_{21} + \beta_{11} \beta_{21} + 2 \beta_{22}\big),\\
\delta_2&= 2 \big(-b_{02} + 2 b_{21} \alpha_{01} + a_{21} \beta_{01} - \alpha_{21} \beta_{01} + \beta_{02} - 2 \alpha_{01} \beta_{21}\big),\\
\delta_3&= 2 \big(a_{02} + b_{11} \alpha_{01} - \alpha_{02} - 2 b_{21} \beta_{01} - 4 \alpha_{11} \beta_{01} - \alpha_{01} \beta_{11} - 2 \beta_{01} \beta_{21}\big),\\
\delta_4&= \alpha_{01} \big(-4 b_{21} - 4 \alpha_{11}\big) + \beta_{01} \big(-2 a_{21} - 2 b_{11} +
2 \alpha_{21} + 2 \beta_{11}\big),\\
\delta_5&= \big(4 b_{21} + 4 \alpha_{11}\big) \beta_{01} + \alpha_{01} \big(-2 a_{21} - 2 b_{11} + 2 \alpha_{21} + 2 \beta_{11}\big),
\end{align*} $\;$ \vspace{-0,7cm}
\begin{align*}
\mu_1&= \pi(\alpha_{11} + \beta_{21}), & \eta_1 & = 2 \big(\beta_{21}-b_{21}\big),\\
\mu_2&= -4\beta_{01},& \eta_2 & = 4 \big(\alpha_{11}+b_{21}\big), \\
\mu_3&= 4\alpha_{01}, & \eta_3 &=  2 \big(\alpha_{21} + \beta_{11} - a_{21}-b_{11}\big).
\end{align*}

Again, using the identities \eqref{eq:id} equation $\Delta_2(r)=0$ is transformed into the polynomial equation $x^3(1+x^4)p_2(x)=0,$ where
\begin{equation}\label{p2}
p_2(x) = \sum_{k=0}^{7} \lambda_k u_k(x),
\end{equation}
with
\begin{align*}
u_0(x)&= x^5, & u_1(x) & = x^4, & u_2(x)&= x^6,& u_3(x) & = x^7, \\
u_4(x)&= 1, & u_5(x) &=  x^2, & u_6(x)&= x^{3}-x^{7}, & u_7(x) &=  x+3x^9,
\end{align*}
and coefficients
\begin{align}
\lambda_0=& -8 \pi \big(a_{12} + b_{22} + \alpha_{12} + (a_{21} - \alpha_{21})(\alpha_{11} + \beta_{21}) + \beta_{22}\big), \nonumber \\
\lambda_1=& 8 \big(3 \beta_{02}-3 b_{02} + 4 b_{21} \alpha_{01} + 3 a_{21} \beta_{01} - 3 \alpha_{21} \beta_{01} - 4 \alpha_{01} \beta_{21}\big), \nonumber \\
\lambda_2=& 24 \big(a_{02} - a_{21} \alpha_{01} -\alpha_{02} + \alpha_{01} \alpha_{21} - 2 \beta_{01} (\alpha_{11} + \beta_{21})\big),\nonumber \\
\lambda_3=& 8 \pi (\alpha_{11} + \beta_{21})^2,\nonumber \\
\lambda_4=& -8 \big(b_{02} + b_{11} \beta_{01} - \beta_{02} - \beta_{01} \beta_{11} + 2 \alpha_{01} (\alpha_{11} + \beta_{21})\big), \label{lambdas} \\
\lambda_5=& 8 \big(a_{02} + b_{11} \alpha_{01} - \alpha_{02} - \alpha_{01} \beta_{11} - 4 \beta_{01} (b_{21} + 2 \alpha_{11} + \beta_{21})\big),\nonumber \\
\lambda_6=& 4 \pi \big(\alpha_{11} + \beta_{21}\big) \big(b_{21} + 2 \alpha_{11} + \beta_{21}\big),\nonumber \\
\lambda_7=& -\pi \big(2 a_{12} + 2 b_{22} + (a_{21} - b_{11} - \alpha_{21} + \beta_{11}) (\alpha_{11} + \beta_{21}) \nonumber \\
&+ 2 (\alpha_{12} + \beta_{22})\big). \nonumber
\end{align}

Considering the ordered set of functions $\mathcal{F}=[u_0,u_1,\ldots,u_7]$ and computing the Wronskians defined in \eqref{Wron} we obtain
\begin{align*}
W_0(x)&= x^5, & W_1(x) & = -x^8, \\
W_2(x)&= -2x^{12}, & W_3(x) & = -12x^{16}, \\
W_4(x)&= -10080x^{12}, & W_5(x) &= -2419200x^9, \\
W_6(x)&= -174182400x^6, & W_7(x) &=-125411328000(1+189x^8).
\end{align*}

From Theorem \ref{ECT}, the maximum number of positive roots of polynomial $p_2$ is seven and the upper bound is reached. From \eqref{lambdas}, it is clear that for any choice of $\lambda$'s it is possible to obtain the corresponding perturbation. This concludes the proof of Proposition \ref{prop2}.
\end{proof}

\begin{example}
In what follows, we exhibit a concrete perturbation presenting seven limit cycles for $|\ep|\neq0$ sufficiently small. Choosing
\begin{align}
\lambda_0&= 1, &
\lambda_1&= -\frac{434699860}{124987243}, \nonumber \\
\lambda_2&= -\frac{18912094}{124987243}, &
\lambda_3&= \frac{5527397195}{874910701}, \label{choice} \\
\lambda_4&= \frac{1929240}{1582117}, &
\lambda_5&= -\frac{613409686}{124987243}, \nonumber \\
\lambda_6&= \frac{11037105503}{1749821402}, &
\lambda_7&= -\frac{6534}{874910701}, \nonumber
\end{align}
the polynomial \eqref{p2} becomes
\[
p_2(x)=-\frac{\left(39204 x^2+1097712x+423361\right)}{1749821402}\prod_{k=1}^7 (x-k).
\]
\begin{figure}[h]
\begin{center}
\begin{overpic}[height=5cm]{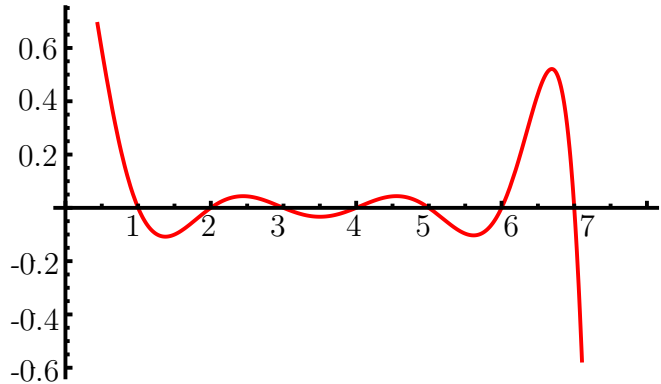}
\put(-2,51){$0.6$} \put(15,23){$1$} \put(27,23){$2$}
\put(-2,43){$0.4$} \put(38,23){$3$} \put(50,23){$4$}
\put(-2,34){$0.2$} \put(61,23){$5$} \put(75,23){$6$}
\put(-3.5,0){-$0.6$} \put(87,23){$7$}
\put(-3.5,8){-$0.4$}
\put(-3.5,17){-$0.2$}\end{overpic}
\caption{Graph of $p_2(x).$}\label{fig3}
\end{center}
\end{figure}
So, $p_2(x)$ has seven positive roots at $1,2,\dots,7$ (see Figure \ref{fig3}).

The system composed by equations \eqref{lambdas} with the parameters given by \eqref{choice} has infinitely many solutions. One of these solutions is given by
$a_{01}=\overline{a_{01}},$ $a_{11}=\overline{a_{11}},$ $a_{21}=\overline{a_{21}},$
$b_{01}=\overline{b_{01}},$ $b_{21}=\overline{b_{21}},$
$\al_{01}=\overline{\al_{01}},$ $\al_{11}=\overline{\al_{11}},$
$\beta_{01}=\overline{\beta_{01}},$
$a_{02}=\overline{a_{02}},$ $a_{12}=\overline{a_{12}},$
$b_{02}=\overline{b_{02}},$ where
\begin{align*}
\overline{a_{01}}&=\frac{822816127288564767 \sqrt{\frac{5527397195 \pi }{1749821402}}}{9860485382218537051},\\
\overline{a_{11}}&= \frac{1}{2} \sqrt{\frac{5527397195}{1749821402 \pi }}-\frac{17688887}{2 \sqrt{9671957909165767390 \pi }},\\
\overline{a_{21}}&= \frac{874936837}{\sqrt{9671957909165767390 \pi }},\\
\overline{b_{01}}&= -\frac{1014941851585915217 \sqrt{\frac{11054794390 \pi }{874910701}}}{29581456146655611153},\\
\overline{b_{21}}&= \frac{17688887}{2 \sqrt{9671957909165767390 \pi }},\\
\overline{\al_{01}}&= \frac{822816127288564767 \sqrt{\frac{5527397195 \pi }{1749821402}}}{9860485382218537051},\\
\overline{\al_{11}}&= -\frac{1}{2} \sqrt{\frac{5527397195}{1749821402 \pi }},\\
\overline{\beta_{01}}&= -\frac{1014941851585915217 \sqrt{\frac{11054794390 \pi }{874910701}}}{29581456146655611153},\\
\overline{a_{02}}&= \frac{26106731320711553594487715355}{103524530135484878237818593012},\\
\overline{a_{12}}&= \frac{874962973}{6999285608 \pi },\\
\overline{b_{02}}&= \frac{1918068234277679948089129635}{17254088355914146372969765502},
\end{align*}
and all other coefficients equal to zero. Notice that conditions \eqref{eq:con1} have been taken into account. So, for $|\ep|\neq0$ sufficiently small, system
\[\begin{array}{l}
\dot x = \left\{\begin{array}{ll}
y +(\overline{a_{01}}+\overline{a_{11}}x+\overline{a_{21}}y)\varepsilon + (\overline{a_{02}}+\overline{a_{12}}x)\varepsilon^2 & \mbox{if } y\geq x^3, \\
y +(\overline{b_{01}}+\overline{b_{21}}y)\varepsilon+\overline{b_{02}}\varepsilon^2  & \mbox{if } y\leq x^3,
\end{array} \right. \vspace{0.2cm}\\ 
\dot y = \left\{\begin{array}{ll}
-x +(\overline{\al_{01}}+\overline{\al_{11}}x)\varepsilon   & \mbox{if } y\geq x^3, \\
-x + \overline{\beta_{01}}\varepsilon & \mbox{if } y\leq x^3
\end{array} \right. \\
\end{array}
\]
has seven limit cycles.
\end{example}

\section*{Acknowledgements}

We thank to the referee for the helpful comments and suggestions.

JLRB has been partially supported by the Brazilian FAPESP grant
2013/24541-0. CAB has been partially supported by the Brazilian FAPESP grant
2013/24541-0 and the Brazilian Capes grant 88881.068462/2014-01. JL has been partially supported by a FEDER-MINECO grant MTM2016-77278-P, a MINECO grant MTM2013-40998-P, and an AGAUR grant number 2014SGR-568. DDN has been partially supported by the Brazilian FAPESP grant 2016/11471-2.

\bibliographystyle{siam}
\bibliography{biblio.bib}

\end{document}